\documentclass[12pt]{paper}

\usepackage{amssymb,amsfonts,amsthm,amsmath}

\usepackage{graphicx, verbatim}
\usepackage{xcolor}

\usepackage[a4paper, hmargin=2.75cm, vmargin={4cm, 4cm}]{geometry}
\usepackage{fancyhdr}
\newtheorem{theorem}{Theorem}
\newtheorem{lemma}{Lemma}

\newtheorem{corollary}{Corollary}

\newtheorem{definition}{Definition}

\def\done{{1\hskip-2.5pt{\rm l}}}

\def\ep{{\varepsilon}}
\newcommand{\card}{\operatorname{card}}
\newcommand{\osc}{\operatorname{osc}}

\newcommand{\bR}{\mathbb R}
\newcommand{\bC}{\mathbb C}
\newcommand{\bZ}{\mathbb Z}

\newcommand{\bT}{\mathbb T}
\newcommand{\bN}{\mathbb N}

\newcommand{\bP}{\mathbb P}

\newcommand{\bE}{\mathbb E}

\DeclareMathOperator*{\triplesum}{\,\sum\,\sum\,\sum\,}

\DeclareMathOperator*{\doublesum}{\,\sum\,\sum\,}

\begin{document}

\title{The ``pits effect'' for entire functions of exponential type and the Wiener spectrum}
\author{Jacques Benatar \and
Alexander Borichev \and Mikhail Sodin}

\author{Jacques Benatar \thanks{Supported by ERC
Advanced Grant 692616}   \and Alexander Borichev
\thanks{Partially supported by a joint grant of Russian Foundation for Basic Research
and CNRS (projects 17-51-150005-NCNI-a and PRC CNRS/RFBR 2017-2019)
and by the project ANR-18-CE40-0035}  \\
\and
Mikhail Sodin
\thanks{Supported by ERC
Advanced Grant 692616 and by ISF Grant 382/15.} }

\maketitle

\begin{abstract} 
Given a sequence $\xi\colon \mathbb Z_+ \to \mathbb C$, 
we find a simple spectral condition  
which guarantees the angular equidistribution of the zeroes of the Taylor series
\[
F_\xi (z) = \sum_{n\ge 0} \xi (n) \frac{z^n}{n!}\,.
\]
This condition yields practically all known instances of random and pseudo-random 
sequences $\xi$ with this property (due to Nassif, Littlewood, Chen--Littlewood,  Levin, 
Eremenko--Ostrovskii,
Kabluchko--Zaporozhets, Borichev--Nishry--Sodin), and provides several new ones. Among them are Besicovitch
almost periodic sequences and multiplicative random sequences. It also conditionally yields
that the M\"obius function $\mu$ has this property assuming ``the binary Chowla conjecture".
\end{abstract}

\section{Introduction and the main result}
Let $F_\xi$ be an entire function of exponential type represented by the
Taylor series
\[
F_\xi (z) = \sum_{n\ge 0} \xi(n) \frac{z^n}{n!}\,, \qquad \xi\colon \bZ_+ \to \bC\,.
\]
As in~\cite{BNS}, we are interested in the influence of the sequence $\xi$ on the 
asymptotic behaviour of the function $F_\xi$, in particular, on the angular
distribution of its zeroes. 

\begin{definition}
The sequence $\xi\colon \bZ_+\to\bC$ is called an $L$-sequence,
if
\begin{equation}\label{eq:log-as}
\frac{\log |F_\xi (tz)|}{t} \underset{t\to\infty}\longrightarrow |z|\,, \qquad {\rm in\ }L^1_{\rm loc}(\mathbb C)\,.
\end{equation}
\end{definition}
In the classical terminology of the entire function theory~\cite{Levin},
condition~\eqref{eq:log-as} means that the function $F_\xi$ has completely regular growth 
in the Levin-Pfluger sense with the indicator diagram being the closed unit disk.

Since the $L^1_{\rm loc}(\mathbb C)$-convergence of subharmonic fucntions
implies convergence in the sense of distributions, and the Laplacian is continuous in the distributional topology, ~\eqref{eq:log-as} yields
\begin{equation}\label{eq:zeroes-as}
\frac1{t}\, \Delta \log |F_\xi (tz)| \underset{t\to\infty}\longrightarrow \Delta |z| = 
{\rm d} r \otimes {\rm d} \theta\,,
\qquad z=re^{{\rm i}\theta}\,,
\end{equation}
in the sense of distributions, with $r{\rm d} r \otimes {\rm d} \theta$ 
being the planar Lebesgue measure. 
Denoting by $n_F(r; \theta_1, \theta_2)$ the number of zeroes (counted with multiplicities)
of the entire function $F$ in the sector $\bigl\{z\colon 0\le |z|\le r, \theta_1 \le \arg (z) < \theta_2 \bigr\}$
and recalling that $\frac1{2\pi}\, \Delta \log |F|$ is the sum of point masses at zeroes of $F$, we
can rewrite~\eqref{eq:zeroes-as} in a more traditional form:
for every $\theta_1<\theta_2$,
\[
n_{F_\xi}(r; \theta_1, \theta_2) = \frac{(\theta_2 - \theta_1 + o(1))\,r}{2\pi},
\qquad r\to\infty\,,
\]
which means that the angular distribution of zeroes is uniform, while the radial one
is proportional to $r$. Littlewood called such behaviour ``the pits effect'' which he
described as follows: ``If we erect an ordinate $|f(z)|$ at the point $z$ of the $z$-plane, then the resulting surface is an exponentially rapidly rising bowl, approximately of revolution, with exponentially small pits going down to the bottom. The zeros of $f$, 
more generally the $w$-points where $f = w$, all lie in the pits for $|z| > R(w)$. 
Finally the pits are very uniformly distributed in direction, and as uniformly distributed in distance as is compatible with the order $\rho$''~\cite[p.~195]{Littlewood2}.

There are no general results providing conditions for a sequence $\xi$ to be an $L$-sequence,
but rather a collection of interesting examples. Most of them required an individual, sometimes, 
quite involved, treatment. These examples include:

\medskip\noindent
(a) random independent identically distributed $\xi (n)$ (Littlewood--Offord~\cite{LO}, Kab\-luch\-ko--Zaporozhets~\cite{KZ});

\medskip\noindent
(b) stationary sequences with a logarithmic decay of the maximal correlation coefficient~\cite[Theorem~3]{BNS} and Gaussian stationary sequences~\cite[Theorem~4]{BNS};

\medskip\noindent
(c) $\xi (n) = e(Q(n))$, where $Q(x)=\sum_{k\ge 2}q_kx^k$ is a polynomial
with real coefficients $q_k$, at least one of which is irrational~\cite[Theorem~1]{BNS},
a special case $\xi (n) = e(q n^2) $ with irrational $q$ was treated earlier, by
Nassif~\cite{Nassif} and Littlewood~\cite{Littlewood}, when $q$ is a quadratic irrationality, and then, by Eremenko--Ostrovskii~\cite{EO}, for arbitrary irrational $q$;

\medskip\noindent
(d) $e(n^\beta)$ with non-integer $\beta>1$, for 
$1<\beta<\frac32$ by Chen--Littlewood~\cite{CL}, and for $\beta\ge \tfrac32$ 
in~\cite[Theorem~2]{BNS};

\medskip\noindent
(e) uniformly almost periodic $\xi (n)$ (Levin\cite[Chapter~VI, \S 7]{Levin}).

\medskip\noindent Here and elsewhere, $e(t)=e^{2\pi {\rm i} t}$.

\medskip
In this work, we find a simple spectral condition for the sequence $\xi$ to be an $L$-sequence, which provides an easy and uniform treatment of the aforementioned
examples (the only exception is a result of Kabluchko and Zaporozhets~\cite{KZ}
pertaining to independent random variables which may have no second moment), 
as well as of several new ones, which were previously out of reach. 
Among them are Besicovitch
almost periodic sequences and multiplicative random sequences. It also conditionally yields
that the M\"obius function $\mu$ has this property assuming ``the binary Chowla conjecture''.

Our spectral condition is based on a notion which was at the heart of 
a generalized harmonic analysis developed by Wiener~\cite[Chapter~IV]{Wiener}:

\begin{definition} Let $\xi\colon \bZ_+\to\bC$.
We say that $\xi$ is a Wiener sequence if for every $k\ge 0$ there exists
a (finite) limit
$$
\rho(k) = \lim_{n\to\infty} \frac1n\, \sum_{0\le s< n}\,
\xi (s)\overline{\xi (s+k)}.
$$
\end{definition}

Each Wiener sequence $\xi$ has a unique spectral measure and its support is called {\em the Wiener spectrum} of $\xi$. The spectral measure is constructed as follows.
Let $\xi$ be a Wiener sequence. Then immediately, $\xi (n) = o(n^{1/2})$, $n\to\infty$. Furthermore, we set
$$
\rho(-k) = \overline{\rho(k)}=\lim_{n\to\infty}\frac1n\,
\sum_{0\le s<n}\overline{\xi (s)}\xi (s+k),\qquad k\ge 1.
$$
Now, given $j,k\ge 0$, we have
$$
\rho(k-j)=\lim_{n\to\infty}\frac1n\,
\sum_{0\le s< n}\, \xi (s+j)\overline{\xi (s+k)}.
$$
For any complex numbers $c_0,\ldots,c_m$ we obtain
\begin{multline*}
\sum_{0\le j,k\le m}\, c_j\overline{c_k}\rho(k-j) =
\lim_{n\to\infty}\frac1n\, 
\sum_{0\le s< n}\doublesum_{0\le j,k\le m} c_j\overline{c_k}\, 
\xi (s+j) \overline{\xi (s+k)}
\\
= \lim_{n\to\infty}\frac1n\, \sum_{0\le s< n}\
\Bigl| \sum_{0\le j\le m}c_j\xi (s+j)\Bigr|^2 \ge 0.
\end{multline*}
Hence, by the Herglotz theorem, there exists a positive measure $\mu$ on the unit circle, the spectral measure of the Wiener sequence $\xi$, such that
$\rho(k)=\widehat\mu(k)$, $k\in\mathbb Z$.

Our main result is as follows.
\begin{theorem} \label{t1} Let $\xi$ be a 
Wiener sequence such that the closed support of its spectral measure is the whole unit circle. Then $\xi$ is an $L$-sequence.
\end{theorem}
It is worth mentioning that one does not need to know the spectral measure $\mu$ in order to check that its closed support has no gaps. It is not difficult to show (see~\cite[Theorem~1]{CKM}) that the sequences of measures 
${\rm d}\mu_n (\theta) = \tfrac1n\, \bigl| \sum_{0\le s < n} \xi (s) e(-s\theta) \bigr|^2\, {\rm d} \theta $
and $ {\rm d}\nu_r (\theta)  = (1-r)|f_\xi (re(-\theta))|^2\, {\rm d}\theta $ with
$f_\xi (z) = \sum_{n\ge 0} \xi (n) z^n$,
converge weakly to the measure $ \mu $ as $n\to \infty $ and $r\to 1$ correspondingly. 
Thus, for instance, it suffices to verify that, for any arc $I$,
$\displaystyle \liminf_{n\to\infty} \mu_n(I) >0$.

\medskip
The proof of Theorem~\ref{t1}
will be given in Section~\ref{sect:proof}; its main idea is similar to the
one used in~\cite{BNS}. Since the sequence $\xi$ does not
grow faster than $\sqrt{n}$, the upper bound $ |F_\xi (z) | \lesssim \sqrt{|z|}\, e^{|z|} $
is a straightforward consequence of the estimates for the Taylor coefficients. A matching lower bound cannot hold everywhere (the function $F_\xi$ has zeroes), but a simple argument
based on the subharmonicity of $\log |F_\xi|$, yields that it suffices to find
a sufficiently dense set of points $z$ at which such a lower bound exists. We cannot explicitly locate such points $z$, instead, in several steps, we estimate from below 
the averages of $|F_\xi|^2$ over sets of the form 
$\{w=te(\varphi)\colon r\le t \le r(1+\delta), |\theta-\varphi|\le \delta  \}$ with 
any $\delta>1$, $\theta\in [-\tfrac12, \tfrac12]$, 
and $r\ge r_\delta$. This estimate will be based on the spectral properties of the sequence $\xi$.

In Section~\ref{sect:examples} we discuss various examples of Wiener sequences that have no  gaps in the support of their spectral measures. Likely, all, or almost all of them are well-known to the experts, however, in combination with Theorem~\ref{t1} they provide new instances of $L$-sequences.
 

\subsubsection*{Acknowledgments}
We thank Alon Nishry for numerous helpful  discussions.

\section{Proof of Theorem~\ref{t1}}\label{sect:proof}

\subsection{Two lemmas on Wiener sequences}

\begin{lemma} \label{lem1} Let $\xi$ be a Wiener sequence. Then there exists a
non-decreasing function $\psi$, $\lim_{x\to\infty}\psi(x)=\infty$, such that
\begin{equation}
\Bigl|  \frac1n\sum_{0\le s< n}\xi (s)\overline{\xi (s+k)}-\rho(k)
 \Bigr|\le \frac{1}{\psi(n)},\qquad n\ge 1,\,|k|\le \psi(n).
\label{15g}
\end{equation}
\end{lemma}

\begin{proof} For every $k$,
$$
\max_{n\ge m}\Bigl|  \frac1n\sum_{0\le s< n}\xi (s)\overline{\xi (s+k)}-\rho(k)
 \Bigr|=\ep(m,k)\to 0,\qquad m\to\infty.
$$
Set $M_0=0$. Given $m\ge 1$, we can find $M_m>M_{m-1}$ such that
$$
\max_{|k|\le m}\ep(M_m,k)\le \frac1m,
$$
and set $\psi(x)=m$, $M_m\le x<M_{m+1}$, $m\ge 0$.
\end{proof}

For a positive function $f$ on $[A, B]\cap \mathbb Z$, we put
$$
\osc_{[A,B]}f\, \stackrel{\rm def}=\, \sum_{A<q<B} |f(q)-f(q-1)|\,.
$$

\begin{lemma} \label{lem3a} Let $\xi$ satisfy \eqref{15g} for some non-decreasing function $\psi$, $\lim_{x\to\infty}\psi(x)=\infty$.
Given integers $0<A<B$ and $|h|\le \psi(A)$ and a positive function $f$ we have
$$
\Bigl| \sum_{A\le k<B} \xi (k)\overline{\xi (k+h)}f(k)-\rho(h)\sum_{A\le k<B} f(k) \Bigr| \le 2 \bigl( f(A)+\osc_{[A,B]}f \bigr)\, \frac{B}{\psi(A)}\,.
$$
\end{lemma}

\begin{proof} Applying summation by parts and then using \eqref{15g}, we get
\begin{align*}
\Bigl| &\sum_{A\le k<B} \xi (k)\overline{\xi (k+h)}f(k) - \rho(h)\sum_{A\le k<B} f(k) \Bigr|
\\ &\le f(A)
\Bigl| \sum_{A\le k<B} \xi (k)\overline{\xi (k+h)} - \rho(h)(B-A)\Bigr| \\
&\quad +\sum_{A< q<B} |f(q)-f(q-1)| \cdot 
\Bigl|\sum_{q\le k<B} \xi (k)\overline{\xi (k+h)}-\rho(h)(B-q)\Bigr| \\ &\le
f(A)\frac{2B}{\psi(A)}+\sum_{A< q<B} |f(q)-f(q-1)|\frac{2B}{\psi(A)} \\
&\le  2(f(A)+\osc_{[A,B]}f)\frac{B}{\psi(A)}\,,
\end{align*}
which proves the lemma.
\end{proof}

\subsection{A lemma on entire functions of exponential type}

\begin{lemma}[Lemma~3.2.1 in \cite{BNS}] \label{lem2} Let $F$ be an entire function of exponential type with 
indicator function
\[
h^{F}(\theta) \stackrel{\rm def}= \limsup_{r\to\infty}
r^{-1}\, \log|F(re(\theta)| \le 1\,.
\] 
Suppose that for every $\theta\in [-\tfrac12, \tfrac12]$
there exists a sequence of $R_j$ such  that $\lim_{j\to\infty}R_j=\infty$, $\lim_{j\to\infty}R_{j+1}/R_j=1$,  and a sequence
of $\theta_j\to\theta$ such that
$$
\liminf_{j\to\infty}\frac{1}{R_j}\log|F(R_je(\theta_j))|\ge 1.
$$
Then, condition~\eqref{eq:log-as} holds for the function $F$, that is, it is of completely regular growth with $h^{F}\equiv 1$.
\end{lemma}

\subsection{Approximating the Taylor series by an exponential sum}
Given $r>0$ and an integer $N$ such that $N \simeq r^{1/2}\log r$, set
\[
\gamma(k,r) =
\frac{k-r}{2k} + \frac{(k-r)^2}{2k} + \frac{(k-r)^3}{3k^2}
\]
and
\[
\widetilde {W}_{r,N}(\theta) =
\sum_{|k-r|\le N}\xi (k)e(k\theta)e^{-\gamma(k,r)},
\]
Furthermore, set $U(r)=e^r/\sqrt{2\pi r}$.

\begin{lemma} \label{lem3} Let $\xi$ be a Wiener sequence. Then
$$
\frac{|F_\xi(re(\theta))|}{U(r)}-|\widetilde {W}_{r,N}(\theta)|=o(1),\qquad r\to\infty,\,\theta\in[0,2\pi].
$$
\end{lemma}


Note that if the Wiener sequence $\xi$ is bounded, then, as in~\cite{BNS},
instead of $\widetilde {W}_{r,N}$ we can use the sum
\[
W_{r,N}(\theta) =
\sum_{|k-r|\le N}\xi (k)\,e(k\theta)\,e^{-|k-r|^2/(2r)}\,,
\]
which makes the proof of Theorem~\ref{t1} somewhat simpler (cf. Lemma~4.1.1 in \cite{BNS}).

\begin{proof}
First, we deal with the sum over $k<r-N$. Denote by $p$ the largest integer smaller than $r-N$.
Since the sequence $(r^n/n!)_n$ increases for $n<r$, we have
$$
\sum_{k<r-N}|\xi (k)|\frac{r^k}{k!}\lesssim \sum_{k<r-N}k^{1/2}\frac{r^k}{k!}\lesssim r^{3/2}\frac{r^{p}}{p!}
$$
By Stirling's formula, we have
\begin{align*}
\frac{r^{3/2}}{U(r)}\frac{r^p}{p!} &\lesssim \frac{r^{3/2}\,r^{p+1/2}e^p}{e^rp^{p+1/2}} \\
&\lesssim r^{3/2}\exp\Bigl(-(r-N)\log\Bigl(1-\frac{N}{r}\Bigr)-N\Bigr) \\
&= r^{3/2}\exp\Bigl(-\frac{N^2}{2r}+O\Bigl(\frac{N^3}{r^2}\Bigr)\Bigr) \\
&= o(1), \qquad r\to\infty.
\end{align*}

Second, we deal with the sum over $k>r+N$. Denote by $p$ the smallest integer larger than $r+N$.
Since the sequence $(n^{2}r^n/n!)_n$ decreases for $n\ge r+1$, we have
$$
\sum_{k>r+N}|\xi (k)|\frac{r^k}{k!}\lesssim 
\sum_{k>r+N}k^{1/2}\frac{r^k}{k!}\lesssim r^{3/2}\frac{r^{p}}{p!}
$$
By Stirling's formula, we have
$$
\frac{r^{3/2}}{U(r)}\frac{r^p}{p!}=o(1),\qquad r\to\infty.
$$

Finally, we turn to the central group of indices.
We have
\begin{align*}
\Bigl|\sum_{|k-r|\le N}\xi (k)e(k\theta) &\Bigl(\frac{r^k}{k!U(r)}- e^{-\gamma(k,r)}\Bigr)\Bigr|^2 \\
&\le
\sum_{|k-r|\le N}|\xi_k|^2\cdot\sum_{|k-r|\le N} \Bigl(\frac{r^k}{k!U(r)}- e^{-\gamma(k,r)}\Bigr)^2 \\
&= o(r)\cdot
\sum_{|k-r|\le N}\Bigl(\frac{\sqrt{2\pi r}\, r^k}{k!e^r}- e^{-\gamma(k,r)}\Bigr)^2.
\end{align*}
Again by Stirling's formula, for $|k-r|\le N$, we have
\begin{align*}
\frac{\sqrt{2\pi r}}{e^r} \cdot \frac{r^k}{k!} &=
\Bigl( 1 + O\Bigl( \frac1{r} \Bigr) \Bigr)\, e^{k-r}\,
\Bigl( \frac{r}{k} \Bigr)^{k+1/2} \\
&= \Bigl( 1 + O\Bigl( \frac1{r} \Bigr) \Bigr)\,
\exp\Bigl( k-r + \bigl(k+1/2\bigr)\log\Bigl(1-\frac{k-r}k \Bigr)\, \Bigr) \\
&= \Bigl( 1 + O\Bigl( \frac1{r} \Bigr) \Bigr)\,
\exp\Bigl( -\gamma(k, r) + O\Bigl( \frac{\log^2\,r}r \Bigr) \, \Bigr) \\
&=\Bigl(1+O\Bigl(\frac{\log^2 r}{r}\Bigr)\Bigr)e^{-\gamma(k,r)},
\end{align*}
and hence,
$$
\sum_{|k-r|\le N}\Bigl(\frac{\sqrt{2\pi r} r^k}{k!e^r}- e^{-\gamma(k,r)}\Bigr)^2\lesssim \frac{\log^5r}{r^{3/2}}.
$$
Finally,
$$
\frac{|F_\xi(re(\theta))|}{U(r)}-|\widetilde {W}_{r,N}(\theta)|=o(1),\qquad r\to\infty\,,
$$
proving the lemma.
\end{proof}


\subsection{A lower bound for the exponential sum
$\widetilde {W}$ (beginning)}
\begin{lemma} \label{lem5} Let $\xi$ be a 
Wiener sequence. Given $\delta>0$, $r>r(\delta)$, and $\theta\in[0,2\pi]$, there exist $r_0\in (r,r+\delta r)$ and $\theta_0\in(\theta-\delta,\theta+\delta)$ such that
for some (every) $N \simeq r^{1/2}\log r$ we have
$$
|\widetilde {W}_{r_0,N}(\theta_0)|\gtrsim r^{1/4}.
$$
\end{lemma}

\begin{proof}[Beginning of the proof of Lemma~\ref{lem5}] Choose a non-negative even function $g\in C^2_0(\mathbb R)$ with support on $[-1/2, 1/2]$ and such that $\int g(\varphi)\,d\varphi=1$.
Fix $N \simeq r^{1/2}\log r$.
It suffices to verify that
$$
X=\int_r^{r+\delta r}\, \int_{-1/2}^{1/2} |\widetilde {W}_{s,N}(\varphi)|^2\,  g(\delta^{-1}(\varphi-\theta))\,{\rm d}s\,{\rm d}\varphi\gtrsim r^{3/2}.
$$

Expanding the square we get
\[
X=\delta \int_r^{r+\delta r} \!\sum_{|k_1-s|\le N}\,\, \sum_{|k_2-s|\le N}\,
\xi(k_1) \overline{\xi (k_2)}\, e((k_1-k_2)\theta)\,
\widehat{g}(\delta(k_2-k_1))\, e^{-\gamma(k_1,s)-\gamma(k_2,s)}\,{\rm d}s.
\]
Set
$$
V_h(t)=\int_r^{r(1+\delta)} \,
\done_{[t-N, t+N]}(s)\, \done_{[t+h-N, t+h+N]}(s)\,
e^{-\gamma(t,s)-\gamma(t+h,s)}\, {\rm d}s
$$
Then
\begin{equation}
X=\delta \sum_{|h|\le 2N} \ \sum_{r-N\le k\le r(1+\delta)+N} \
\xi (k)\overline{\xi (k+h)}\, e(-h\theta)
\widehat{g}(\delta h)V_h(k).
\label{sl4}
\end{equation}
To apply Lemma~\ref{lem3a}, we need to estimate the quantities
$ \max_{[A, B]} V_h $, $\osc_{[A, B]} V_h$, and
$\sum_{A\le k < B} V_h(k) $ with $A=r-N$ and $B= r(1+\delta)+N$.
This will be done next in a series of lemmas. In what follows, we always assume
that the value of $r$ (and hence of $t$) is large enough.

\subsection{Estimating the integral $V_h(t)$}
The first estimate is a crude upper bound:
\begin{lemma}\label{lemma:upper_bound}
Let $|h|\le 2N$. Then
\[
\sup_{[r-N, r(1+\delta)+N]}V_h\lesssim r^{1/2}.
\]
\end{lemma}
\begin{proof}
Fix $ r-N \le t \le r(1+\delta)+N $.
Then, for $|s-t|\le N$, we have
\[
\Bigl| \gamma (t, s) - \frac{(t-s)^2}{2t} \Bigr| \lesssim \frac{N}r + \frac{N^3}{r^2}\,
\stackrel{N \simeq \sqrt{r}\, \log r}\lesssim\,
\frac{\log^3 r}{\sqrt{r}}\,,
\]
and similarly, for $|s-(t+h)|\le N$, we have
\[
\Bigl| \gamma (t+h, s) - \frac{(t+h-s)^2}{2(t+h)} \Bigr|
\lesssim \frac{\log^3 r}{\sqrt{r}}\,.
\]
Hence,
\[
V_h(t) \lesssim \int_\bR
\exp\Bigl( -\frac{(t-s)^2}{2t} \Bigr)\, {\rm d}s
\, \stackrel{t\simeq r }\lesssim\,  \sqrt{r}\,,
\]
completing the proof. \end{proof}

\medskip
This lemma (combined with the decay of the Fourier transform
$\widehat{g}(\delta h)$) will allow us to make the sum in $h$ in~\eqref{sl4}
much shorter, cutting it from $|h|\le 2N$ to $|h|\le \log r$. Hence, it
will suffice to estimate $\sum_k V_h(k)$ and the oscillation of $V_h$
only for $|h|\le \log r$.

\medskip In what follows, we assume that $|h|\le \log r$.
Our next goal is to simplify the integrand in the definition of $V_h$.
First, we note that, for $|s-t|\le N$, we have
\[
e^{-\gamma (t, s)} = \Bigl( 1 + Q(t^{-1}, t-s) +
O\Bigl( \frac{\log^q r }{r^{3/2}} \Bigr) \Bigr)\, e^{-(t-s)^2/(2t)}\,,
\]
where $Q\in\bR[x, y]$ is a polynomial which consists of the terms $x^\ell y^m$ with
$m<2\ell$, and $q$ is a positive number. The exact form of $Q$
($Q(x, y) = - \tfrac12 xy + \tfrac18 x^2 y^2 - \tfrac13 x^2 y^3
+ \tfrac16 x^3 y^4 + \tfrac1{18} x^4 y^6$) and the value of $q$ ($q=9$)
are of no importance in our analysis. Replacing $t$ by $t+h$, we see that,
for $|s-(t+h)|\le N$, we have
\begin{multline*}
e^{-\gamma (t+h, s)} = \Bigl( 1 + Q((t+h)^{-1}, t+h-s) +
O\Bigl( \frac{\log^q r }{r^{3/2}} \Bigr) \Bigr)\, e^{-(t+h-s)^2/(2t)} \\
= \Bigl( 1 + P(h, t^{-1}, t-s) +
O\Bigl( \frac{\log^p r }{r^{3/2}} \Bigr) \Bigr)\, e^{-(t-s)^2/(2t)}\,,
\end{multline*}
where $P\in\bR[h, x, y]$ is a polynomial which consists of the terms
$h^k x^\ell y^m$ with $m < 2\ell$,
and $p$ is a positive number. Thus,
\begin{align*}
V_h(t)&=\int_r^{r(1+\delta)} \,
\done_{[t-N, t+N]}(s)\, \done_{[t+h-N, t+h+N]}(s) \times \\
\, &\quad \times
\Bigl( 1 + P_1(h, t^{-1}, t-s) + O\Bigl( \frac{\log^{p} r }{r^{3/2}} \Bigr)\Bigr)\,
e^{-(t-s)^2/t}\, {\rm d}s \\
&= \int_r^{r(1+\delta)} \,
\Bigl( 1 + P_1(h, t^{-1}, t-s) + O\Bigl( \frac{\log^{p} r }{r^{3/2}} \Bigr)\Bigr)\,
e^{-(t-s)^2/t}\, {\rm d}s + O\bigl( e^{-c\log^2 r} \bigr) \\
&=\sqrt{t}\,
\int_{(r-t)/\sqrt{t}}^{(r(1+\delta)-t)/\sqrt{t}}\,
\bigl( 1 + P_1(h, t^{-1}, -u\sqrt{t}) \bigr)\,
e^{-u^2}\, {\rm d}u + O\Bigl( \frac{\log^{p} r }r \Bigr)
\quad (s=t+u\sqrt{t})\,.
\end{align*}
Here, $P_1\in\bR[h, x, y]$ is a polynomial of the same structure as $P$, and the positive
integer value $p$ may vary from line to line.

\medskip
In what follows, we will separate three ranges of values of $t$: the central part
$r+N \le t \le r(1+\delta)-N$, and the edges: $r-N \le t \le r+N$ and $r(1+\delta)-N \le t
\le r(1+\delta)+N $.

\medskip\noindent\underline{$ r+N \le t \le r(1+\delta)-N$}: Then,
\[
[-c\log r, c \log r] \subset [(r-t)/\sqrt{t}, (r(1+\delta)-t)/\sqrt{t}]\,,
\]
so we can replace the integration over the interval
$[(r-t)/\sqrt{t}, (r(1+\delta)-t)/\sqrt{t}]$ by the integration over the whole real axis,
the error we make is $O\bigl( e^{-c\log^2 r}\bigr)$.
We immediately conclude that
\begin{equation}\label{eq:V_h_asymptotics}
V_h(t) = (\sqrt{\pi}+o(1))\sqrt{t}, \qquad r+N \le t \le r(1+\delta)-N,\, |h|\le \log r,\,r\to \infty\,.
\end{equation}
This yields
\begin{lemma}\label{lemma_sum}
Let $|h|\le \log r$.
Then
\[
\sum_{r-N \le k < r(1+\delta)+N}\, V_h(k) = (A_\delta + \ep (r) )r^{3/2}\,,
\qquad r\to\infty\,,
\]
with $A_\delta=\frac23\sqrt{\pi}\bigl((1+\delta)^{3/2}-1\bigr)$, and $\ep (r)$
monotonically decreasing to $0$.
\end{lemma}

\begin{proof}
We split the sum into three parts:
\begin{multline*}
\sum_{r-N \le k < r(1+\delta)+N}\, V_h(k) \\
=
\Bigl( \, \sum_{r-N \le k < r+N}\,
+ \, \sum_{r+N \le k < r(1+\delta)-N}\,
+ \, \sum_{r(1+\delta)-N < k < r(1+\delta)+N}\, \Bigr) V_h(k)\,.
\end{multline*}
By Lemma~\ref{lemma:upper_bound}, the first and the third sum are
$  \lesssim r^{1/2} N \lesssim r \log r $. In the central sum we
use the asymptotics~\eqref{eq:V_h_asymptotics}. \end{proof}

\medskip Next, we will show that
\begin{equation}\label{eq:osc1}
\sum_{r+N \le k \le r(1+\delta)-N} |V_h(k+1)-V_h(k)| \lesssim \sqrt{r}.
\end{equation}
To this end, we will re-write the asymptotics~\eqref{eq:V_h_asymptotics} with a more accurate
error term. We have
\begin{multline*}
V_h(t) = \sqrt{t}\, \int_{-\infty}^\infty
\Bigl( 1 +\, \triplesum_{k, \ell, m\colon m<2\ell}\, C_{k, \ell, m} h^k t^{m/2-\ell} u^m
\Bigr)\, e^{-u^2}\, {\rm d}u + O\Bigl( \frac{\log^p r}r \Bigr) \\
= \sqrt{\pi t} + \triplesum_{k, \ell, m\colon m < 2\ell}\,
C'_{k, \ell, m} h^k t^{(m+1)/2-\ell} + O\Bigl( \frac{\log^p r}r \Bigr)\,.
\end{multline*}
Each term $f(t)$ in the triple sum on the RHS satisfies
\[
\max_{[r+N, r(1+\delta)-N]}\ |f'| = O_\ep (r^{-3/2+\ep})\,, \qquad r\to\infty\,,
\]
which yields~\eqref{eq:osc1}.

\medskip It remains to prove counterparts of the estimate~\eqref{eq:osc1} at the edges of the range of $t$.

\medskip\noindent\underline{$r - N \le t \le r + N$}:
In this case, extending the upper limit of the integrals to $+\infty$, we get
\begin{multline*}
V_h(t)
= \sqrt{t}\, \int_{(r-t)/\sqrt{t}}^\infty\, e^{-u^2}\, {\rm d}u \, \\
+ \, \triplesum_{k, \ell, m\colon m < 2\ell}\, C_{k, \ell, m}
h^k t^{(m+1)/2-\ell} \int_{(r-t)/\sqrt{t}}^\infty\, u^m e^{-u^2}\, {\rm d}u \,
+ \, O\Bigl( \frac{\log^p r}r \Bigr)\,.
\end{multline*}
Observe that the function $t\mapsto (r-t)/\sqrt{t}$ decreases. Therefore,
the first term on the RHS increases with $t$ and satisfies
\[
\sum_{r - N \le k \le r + N} |f(k)-f(k-1)| \le f(r+N) \lesssim \sqrt{r}\,.
\]
A similarly straightforward inspection shows that for other terms we have the even 
better estimate
\[
\sum_{r - N \le k \le r + N} |f(k)-f(k-1)| \lesssim \log^p r\,.
\]
We conclude that
\begin{equation}\label{eq_osc2}
\sum_{r - N \le k \le r + N} |V_h(k)-V_h(k-1)| \lesssim \sqrt{r}\,.
\end{equation}

\medskip\noindent\underline{$r(1+\delta)-N \le t \le r(1+\delta)+N$}:
In this case, extending the lower limit of the integrals to $-\infty$, we have
\begin{multline*}
V_h(t)
= \sqrt{t}\, 
\Bigl(\, \int_{-\infty}^\infty - \int_{(r(1+\delta)-t)/\sqrt{t}}^\infty\, \Bigr)
e^{-u^2}\, {\rm d}u \, \\
+ \, \triplesum_{k, \ell, m\colon m < 2\ell}\, C_{k, \ell, m}
h^k t^{(m+1)/2-\ell} 
\Bigl( \int_{-\infty}^\infty - 
\int_{(r(1+\delta)-t)/\sqrt{t}}^\infty\, \Bigr) 
u^m e^{-u^2}\, {\rm d}u \, + \, O\Bigl( \frac{\log^p r}r \Bigr) 
\end{multline*}
which, similarly to the previous case, yields
\begin{equation}\label{eq_osc3}
\sum_{r(1+\delta) - N \le k \le r(1+\delta) + N} |V_h(k)-V_h(k-1)| \lesssim \sqrt{r}\,.
\end{equation}

\medskip Combining estimates~\eqref{eq:osc1}, ~\eqref{eq_osc2}, and~\eqref{eq_osc3},
we obtain

\begin{lemma}\label{lemma:osc_V_h}
Let $|h|\le \log r$.
Then
\[
\sum_{r-N \le k < r(1+\delta)+N}\, | V_h(k) - V_h(k-1)| \lesssim \sqrt{r}\,.
\]
\end{lemma}

\subsection{Completing the proof of the lower bound for $\widetilde{W}$}

Let us return to the proof of Lemma~\ref{lem5}.
We use that for $|h|\le 2N$,
\begin{multline}
\, \sum_{r-N\le k\le r(1+\delta)+N}\,
|\xi (k)\overline{\xi (k+h)}| \\
\le
\Bigl( \, \sum_{r-N\le k\le r(1+\delta)+N}\,
|\xi (k)|^2\Bigr)^{1/2} \cdot \Bigl(\,
\sum_{r-N\le k\le r(1+\delta)+N}\, |\xi (k+h)|\Bigr)^{1/2}\lesssim r\,.
\label{sl6}
\end{multline}

Choose a large $H \le \min( \log r, \psi(r/2), 1/\ep(r))$ with the functions $\psi$ and $\ep$ as in Lemma~\ref{lem2} and Lemma~\ref{lemma_sum}. By \eqref{sl4} we have that 
\begin{multline*}
\Bigl|X-\delta\sum_{|h|\le H}\,\,\sum_{r-N\le k\le r(1+\delta)+N}\,
\xi (k)\overline{\xi (k+h)}e(-h\theta)
\widehat{g}(\delta h)V_h(k)\Bigr|\\
\le \delta \sum_{H<|h|\le 2N}\,\,\sum_{r-N\le k\le r(1+\delta)+N}\,
|\xi (k)\overline{\xi (k+h)}|\cdot
|\widehat{g}(\delta h)|\cdot V_h(k).
\end{multline*}
Since $|\widehat{g}(\delta h)|\lesssim (\delta h)^{-2}$,
using Lemma~\ref{lemma:upper_bound} and estimate~\eqref{sl6}, we obtain the bound 
\begin{multline*}
\Bigl|X-\delta\sum_{|h|\le H}\,\,\sum_{r-N\le k\le r(1+\delta)+N}
\xi (k)\overline{\xi (k+h)}e(-h\theta)
\widehat{g}(\delta h)V_h(k)\Bigr| \\
\lesssim \delta\sum_{H<|h|\le 2N}\frac{r^{3/2}}{(\delta h)^2}
\simeq \frac{r^{3/2}}{\delta H}\,.
\end{multline*}
Next, by Lemma~\ref{lem3a}, we have that 
\begin{multline*}
\Bigl|\, \sum_{r-N\le k \le r(1+\delta)+N}\, \xi (k) \overline{\xi (k+h)}\, V_h(k)
- \rho(h) \, \sum_{r-N\le k \le r(1+\delta)+N}\, V_h(k)
\, \Bigr| \\
\lesssim \frac{r}{\psi(r/2)} \cdot \osc_{[r-N, r(1+\delta)+N]}\, V_h
\end{multline*}
By Lemma~\ref{lemma:osc_V_h}, the RHS is
\[
\lesssim \frac{r^{3/2}}{\psi(r/2)}\,.
\]
Thus, recalling Lemma~\ref{lemma_sum}, we get that 
\[
\Bigl|\, \sum_{r-N\le k \le r(1+\delta)+N}\, \xi (k) \overline{\xi (k+h)}\, V_h(k)
- A_\delta \rho(h) r^{3/2}\, \Bigr| \lesssim
\Bigl(\, \ep (r) + \frac1{\psi(r/2)}\, \Bigr)\, r^{3/2}\,,
\]
whence, using once again that $|\widehat{g}(\delta h)|\lesssim (\delta h)^{-2}$,
we obtain the estimate 
$$
\Bigl| X-\delta A_\delta r^{3/2}\,
\sum_{|h|\le H}\rho(h)e(-h\theta)\widehat{g}(\delta h)\, \Bigr|
=O\Bigl(\, \frac{r^{3/2}}{\delta H} \, \Bigr)\,, \qquad r\to\infty.
$$
Arguing as in the proof of Lemma~7.3.1 of \cite{BNS}, we use that
$$
\sum_{h\in\mathbb Z}\rho(h)e(-h\theta)
\widehat{g}(\delta h)
$$
is the density of the convolution  $\mu \ast g_\delta$ at the point $-\theta$, where
$g_\delta(\varphi) =\delta^{-1} g(\delta^{-1}\varphi)$, $ |\varphi|\le 1$.
Since the support of $\mu$ is the whole circle and the function $g$ is non-negative, this value is positive.
Hence, for some $c=c(\delta)>0$, we have $X=(c+o(1))r^{3/2}$, $r\to\infty$, and
the lemma is proved.
\end{proof}

\medskip
Theorem~\ref{t1} follows now from Lemmas~\ref{lem2}, \ref{lem3}, and \ref{lem5}.
\hfill $\Box$

\section{Examples to Theorem~\ref{t1}}\label{sect:examples}

\subsection{Ergodic stationary processes on $\bZ$}
Let $\xi\colon \bZ_+\to\bC$ be a stationary process on $\bZ$ whose elements have
finite second moment. Then the sequence $m\mapsto \bE\bigl[ \xi(0)\overline{\xi(m)}\bigr]$
is positive-definite and may therefore be expressed as the Fourier transform of a non-negative measure
$\mu$ on the unit circle, called {\em the spectral measure} of the process $\xi$.
By the Birkhoff ergodic theorem, almost surely, the limits
\begin{equation}\label{eq:Birkhoff}
\lim_{n\to\infty} \frac1n\, \sum_{0\le s < n} \xi(s)\overline{\xi(s+k)}
\end{equation}
exist for every $k$, that is, almost every realization of $\xi$ is a Wiener sequence.
Generally speaking, the spectral measure of a realization, as well as its closed
support, are random, but if the process $\xi$ is {\em ergodic}, then the limits~\eqref{eq:Birkhoff} are not random and coincide with the Fourier coefficients
$\widehat{\mu}(k)$. That is, for ergodic stationary processes on $\mathbb Z$, for
almost every realization, its spectral measure coincides with the spectral measure of
the processes, cf.~\cite[Section~2]{BSW}. Combining this discussion with our Theorem~1,
we obtain

\begin{corollary}\label{cor:stationary_ergodic}
Let $\xi$ be an ergodic stationary process on $\bZ$ with no gaps in its spectrum.
Then, almost surely, $\xi$ is an $L$-sequence.
\end{corollary}

This corollary removes unnecessary assumptions in Theorem~3 from~\cite{BNS} where it was assumed that the sequence $\xi$ is bounded and has strong mixing properties.

It is worth mentioning that, as follows from~\cite[Theorem~4]{BNS}, for
Gaussian stationary processes on $\bZ$, Corollary~\ref{cor:stationary_ergodic} holds
{\em without} the ergodicity assumption.

\subsection{Besicovitch almost-periodic sequences}

For functions (sequences) $s\colon \bZ_+\to \bC$, define a semi-norm 
\[
\| \xi \|^2 = \limsup_{n\to\infty} \frac1n\, \sum_{0\le s < n} |\xi(s)|^2\,.
\]
Two sequences $\xi$ and $\widetilde\xi$ are said to be {\em equivalent} if
$\| \xi - \widetilde\xi \|=0$.
For instance, any sequence $\xi$ with $\xi(n) = o(1)$ as $n\to\infty$,
is equivalent to the zero function.
By $\ell^2$ we denote the linear space of the equivalence classes equipped with
the norm $\|\cdot\|$. One can show that this space is complete.

By $\mathcal P$ we denote the linear hull of the exponentials, i.e.,
the elements of $\mathcal P$  are the finite linear combinations of exponentials
\[
P(n) = \sum_{\lambda\in\Lambda} c_\lambda e_\lambda (n)\,,
\]
where $e_\lambda (n) = e(\lambda n)$,
and $\Lambda\subset\bT$ is a finite set, {\em the spectrum of $P$}.

\begin{definition}\label{def:Besicovitch}
The Besicovitch space $B^2$ is the closure of $\mathcal P$ in $\ell^2$.
\end{definition}

We say that a sequence $\xi\colon \bZ_+\to\bC$ has a mean if the following
limit exists (and is finite):
\[
\mathsf M(\xi) = \lim_{n\to\infty} \frac1n\, \sum_{0\le s<n} \xi(s)\,.
\]
The following two facts are relatively straightforward:
\begin{enumerate}
\item For any $\xi\in B^2$ and any $\lambda\in\bT$, the Fourier coefficient of $\xi$ at $\lambda$, 
i.e. 
$ \widehat{\xi}(\lambda)=\mathsf M(\xi e_{-\lambda})$ is well defined.
\item For any $\xi, \eta \in B^2$, the mean $\mathsf M(\xi\,\overline{\eta})$
exists. It defines the scalar product
$ \langle \xi, \eta \rangle = \mathsf M(\xi\,\overline{\eta})$ in $B^2$.
Hence, the Cauchy--Schwartz inequality
$|\langle \xi, \eta \rangle | \le \|\xi\| \cdot \|\eta \| $ holds for every
$\xi, \eta \in B^2$.
\end{enumerate}
Noting that the translations $\xi_m(n)=\xi(n+m)$, $m\in \bN$,
preserve $B^2$, we conclude that {\em any $B^2$-sequence is a Wiener sequence}.
Furthermore,
\begin{enumerate}
\item[3.] The set $\Lambda_\xi = \{\lambda\colon \widehat{\xi}(\lambda)\ne 0\}$ (called
{\em the spectrum} of~$\xi$) is countable, and the Parseval identity
$ \| \xi \|^2  = \sum_{\lambda\in\Lambda_\xi} |\widehat{\xi}(\lambda)|^2 $ 
holds.
\end{enumerate}
Then, the spectral measure $\mu$ of $\xi$ is nothing but
$ \sum_{\lambda\in\Lambda_\xi} |\widehat{\xi}(\lambda)|^2\, \delta_\lambda $, where $\delta_\lambda$ is the unit point mass at $\lambda$. 

Combining these preliminaries\footnote{
The reader will find more details, for instance, in~\cite[Sections 2.6, 3.4, 4.2 ]{Cord} (where the Besicovitch 
almost-periodic functions on $\bR$ are treated)
as well as in~\cite[Section~3]{Bellow}. 
} with Theorem~1, we arrive at

\begin{corollary}\label{cor:Besicovitch}
Any $B^2$-sequence whose spectrum is dense in the unit circle
is an $L$-sequence.
\end{corollary}

If the sequence $\xi$ is uniformly almost-periodic, then the function
$F_\xi$ has a completely regular growth in the sense of Levin-Pfluger with the
conjugated indicator diagram being the convex hull of the spectrum $\Lambda_\xi$.
This is a theorem of Levin~\cite[Section~VI.5]{Levin} (for a relatively self-contained
proof see~\cite[Theorem~5]{BNS}). If the spectrum $\Lambda_\xi$ is dense in $\bT$, then
the conjugated indicator diagram is the closed unit disk. In this special case, Corollary~\ref{cor:Besicovitch} provides a much stronger result\footnote{
For instance, uniform almost-periodic sequences cannot attain finitely many values, while
there are plenty of $B^2$-sequences that attain only finitely many values and have a spectrum 
that is dense in $\bT$. One of the simplest examples is, probably, $\xi (n) = (-1)^{[\alpha n]}$ with irrational $\alpha$.
}.
On the other hand, it is not difficult to construct a bounded $B^2$-sequence $\xi$ such that
the closure $\overline{\Lambda_\xi}\subsetneqq \bT$, but still the indicator diagram
of $F_\xi$ is the closed unit disk (i.e., the Phragm\'en-Lindel\"of indicator
$h_{F_\xi} \equiv 1$). This leaves no hope to extend the general case of Levin's theorem to $B^2$-sequences.

\subsection{Unimodular pseudo-random sequences}

\subsubsection{}

\begin{corollary}
Let $(a_n)_{n\ge 1}$ be a sequence of positive integers such that for every $k\ge 1$, the sequence $(a_{n+k}-a_n)_{n\ge 1}$ consists of distinct numbers. Then for almost every real number $x$, the sequence
$(e(a_nx))_{n\ge 1}$ is a Wiener sequence, its spectral measure being Lebesgue measure on the unit circle, and, hence, this sequence is an $L$-sequence.

In particular, for almost every real number $t$ and for every natural number $a\ge 2$, the sequence $(e(a^nt))_{n\ge 1}$ is an $L$-sequence.
\end{corollary}

Indeed, we fix $k\ge 1$, consider the sum
\[
\sum_{0\le s\le n} e(a_s x)\cdot \overline{e(a_{s+k}x)} =
\sum_{0\le s \le n} e((a_s-a_{s+k})x)\,,
\]
and apply the classical theorem of Weyl
(see, for instance,~\cite[Theorem~4.1 in Chapter~1]{KN}),
which says that, given a sequence $(b_n)$ of distinct integers (in our case,
$b_n=a_n-a_{n+k}$), for almost every real number $x$,
\[
\sum_{0\le s \le n} e(b_s x) = o(n)\,, \qquad n\to\infty\,.
\]

\subsubsection{}

\begin{corollary}
Let $\alpha>0$, $\beta>1$, $\beta\not\in\mathbb N$. Then the sequence
$(e(\alpha n^\beta))_{n\ge 1}$ is a Wiener sequence, its spectral measure being Lebesgue measure on the unit circle, and, hence, this sequence is an $L$-sequence.
\end{corollary}

For $1<\beta\le \tfrac32$, this follows from a result of Chen and Littlewood~\cite{CL}; 
for $\beta\ge \tfrac32$, $\beta\notin \bN$, this was proven in~\cite[Theorem~2]{BNS}.
The techniques used in these cases were quite different and the case
$1<\beta \le \tfrac32$ required a rather elaborate argument\footnote{
Note that Chen and Littlewood proved a much finer result: they found
an asymptotic location of zeroes of the entire function $F_\xi$.
}. Theorem~\ref{t1} combined with classical van der Corput estimates
of exponential sums allows us to treat both cases in a simple and uniform
manner.

Let us also mention that for $0<\beta<1$, the sequence $(e(\alpha n^\beta))_{n\ge 1}$ is a Wiener sequence, its spectral measure being the unit point mass at $1$. 
It is not difficult to show that the corresponding entire function $F_\xi$ has the Phragm\'en--Lindel\"of indicator  $ h_{F_\xi}(\theta) = \cos_+ \theta $ (cf. \cite[Theorem 2]{EO}), which is obviously incompatible with $\xi$  
being an $L$-sequence.

\begin{proof} We need to show that for every $m\ge 1$,
$$
\lim_{n\to\infty}\frac1n\sum_{0\le s< n}e(\alpha (s+m)^\beta-\alpha s^\beta)=0.
$$
Set $\phi(t)=\alpha (t+m)^\beta-\alpha t^\beta$.

For $1<\beta<2$ we apply a van~der~Corput estimate~\cite[Theorem~2.7 in Chapter~1]{KN}.
We have
$$
\Bigl|\frac1n\sum_{0\le s< n}e(\phi(s))\Bigr|\lesssim \frac1{n|\phi''(n)|^{1/2}} \simeq n^{(1-\beta)/2},\qquad n\to\infty.
$$

For $\beta>2$ we use another, more elaborate,
van~der~Corput bound~\cite[Theorem~5.13]{Tit}.
Choose an integer $k\ge 2$ such that $k<\beta<k+1$ and fix $M=2^k$.
Since
$$
n^{\beta-k-1}\lesssim \phi^{(k)}(t)\lesssim n^{(k+1-\beta)/4}n^{\beta-k-1},\qquad n^{3/4}\le t\le n,
$$
we have
\begin{multline*}
\Bigl|\sum_{0\le s< n}e(\phi(s))\Bigr|\le n^{3/4}+\Bigl|\sum_{n^{3/4}\le s< n}e(\phi(s))\Bigr|\\ \lesssim
n^{3/4}+n^{(k+1-\beta)/M}\cdot n\cdot n^{(\beta-k-1)/(M-2)}+n^{1-(4/M)}\cdot  n^{(k+1-\beta)/(M-2)} = o(n)
\end{multline*}
as $n\to\infty$.
\end{proof}

We note that, likewise, one can show that if $Q(x)=\sum_{k=2}^d q_k x^k$ is a polynomial with
real coefficients $q_k$ and at least one of the coefficients is irrational, then the sequence
$(e(Q(n))_{n\ge 1}$ is a Wiener sequence whose spectral measure is the Lebesgue measure
on the unit circle. Therefore, this sequence is an $L$-sequence. The latter conclusion is
Theorem~1 from~\cite{BNS}. 

\subsection{Arithmetic multiplicative functions}
Recall that an arithmetic function
$f\colon\bN\to\bC$ is called {\em completely multiplicative}
if $f(n)f(m)=f(nm)$ for all natural $n, m$,  and
{\em multiplicative} if $f(n)f(m)=f(nm)$ for all mutually prime natural $n,m$.

\subsubsection{The M\"{o}bius function}
The classical M\"{o}bius function $\mu$ is defined by
\[
\mu (n) =
\begin{cases}
\quad 1, & n {\rm\ square-free\ with\ an\ even\ number\ of\ prime\ factors;} \\
-1, & n {\rm\ square-free\ with\ an\ odd\ number\ of\ prime\ factors;} \\
\quad 0, &  {\rm\ otherwise.}
\end{cases}
\]
One instance of the well-known Chowla conjecture (the so called ``binary Chowla conjecture'') asserts that, for every $k\ge 1$,
\begin{equation}\label{eq:Chowla}
\sum_{1\le s\le n}\mu(s)\mu(s+k) = o(n),\qquad n\to\infty,
\end{equation}
that is, $\mu$ is a Wiener sequence whose spectral measure is the Lebesgue measure.
Thus, applying Theorem~\ref{t1}, we get

\begin{corollary} Assuming Chowla's conjecture~\eqref{eq:Chowla}, the sequence
$(\mu(n))_{n\ge 1}$ is an $L$-sequence.
\end{corollary}
The question on the asymptotic behaviour of the Taylor series
of the form
\[
\sum_{n\ge 1} \mu(n)\, a_n z^n
\]
was raised by Paul L\'evy in~\cite[Section~18]{Levy}. He hinted that the positive answer to the 
Riemann hypothesis would provide information on this asymptotic behaviour.

\subsubsection{Random multiplicative functions}
Denote by $\mathsf P$ the set of prime numbers.
Let $(X_p)_{p\in \mathsf P}$ be a sequence of independent Steinhaus random variables
(i.e., uniformly distributed on the unit circle), and let $(Y_p)_{p\in \mathsf P}$ be a sequence of independent Rademacher random variables
(i.e., taking the values $\pm 1$ with equal probability).
We consider two random functions: {\em a Steinhaus random completely multiplicative
function} $\xi_S$ determined by
$\xi_{\tt S}(p) = X_p$, $p\in\mathsf P$, that is,
\[
\xi_{\tt S}(n) = \prod_{p^a\,\|\,n} (X_p)^a,
\]
and {\em a Rademacher random multiplicative function} $\xi_{\tt R} $ determined by
$\xi_{\tt R}(p) = Y_p$, $p\in\mathsf P$, on the square-free indices $p$ and $0$ elsewhere,
that is,
\[
\xi_{\tt R}(n) =
\begin{cases}
\prod_{p\, | \, n} X_p & n {\rm\ is\ square-free;} \\
0 & {\rm \ otherwise.}
\end{cases}
\]
The function $\xi_{\tt S}$ randomizes the prime factorization of integers, while
the function $\xi_{\tt R}$ is a randomized version of the M\"obius function.
The study of these random multiplicative functions has a long history and goes back,
at least, to Wintner~\cite{Wintner}.

\begin{theorem} \label{t2} The sequences $\xi_{\tt S}$ and $\xi_{\tt R}$
are almost surely Wiener sequences, their spectral measures being Lebesgue measure on the unit circle, and, hence, these sequence are almost surely $L$-sequences.
\end{theorem}

The proof is based on the following lemma:
\begin{lemma} \label{l11} Given $k\ge 1$, we have
\begin{align*}
\mathbb E\, \bigl| \sum_{m\le s\le\lambda m}\,
\xi_{\tt S}(s)\overline{\xi_{\tt S}(s+k)}\, \bigr|^2
&= O(m^{7/4})\,, \\
\mathbb E\, \bigl[ \sum_{m\le s\le\lambda m}\,
\xi_{\tt R}(s) \xi_{\tt R}(s+k)\, \bigr]^2
&=O(m^{7/4})\,,
\end{align*}
uniformly in $\lambda\in[1,2]$.
\end{lemma}

\begin{proof}
Since
$$
\mathbb E\, \Bigl|\,\sum_{m\le s\le\lambda m}\,
\xi_{\tt S}(s)\overline{\xi_{\tt S}(s+k)}\,\Bigr|^2
=
\doublesum_{m\le s_1,s_2\le\lambda m}\,
\mathbb E\, \bigl( \xi_{\tt S}(s_1)\overline{\xi_{\tt S}(s_1+k)}\, \overline{\xi_{\tt S}(s_2)}\, \xi_{\tt S}(s_2+k)\, \bigr),
$$
and the Steinhaus random variables $X_p$, $p\in \mathsf P$, are independent,
after a minute reflection we convince ourselves that
$$
\mathbb E\, \Bigl|\, \sum_{m\le s\le\lambda m}\,
\xi_{\tt S}(s)\overline{\xi_{\tt S}(s+k)} \, \Bigr|^2 \le \card(\mathsf B),
$$
where
$$
\mathsf B = \bigl\{ (s_1,s_2)\in[m,\lambda m]\times[m,\lambda m] : s_1(s_1+k)s_2(s_2+k)\text{\ is a square}\bigr\}.
$$
Similarly,
$$
\mathbb E \Bigl[\, \sum_{m\le s\le\lambda m}\,
\xi_{\tt R}(s)\xi_{\tt R}(s+k)\, \Bigr]^2 = \card(\mathsf B)\,.
$$
It remains to estimate the cardinality of the set $\mathsf B$.

Given $(s_1,s_2)\in[m,\lambda m]\times[m,\lambda m]$, set $d_1=\gcd(s_1,s_1+k)$, $d_2=\gcd(s_2,s_2+k)$. Then $d_1$ and $d_2$
are divisors of $k$. Next, let $e_1^2,f_1^2,e_2^2,f_2^2$ be the largest square divisors of $s_1/d_1$, $(s_1+k)/d_1$,
$s_2/d_2$, and $(s_2+k)/d_2$. Then $(s_1,s_2)\in\mathsf B$ if and only if
$$
\frac{(s_1/d_1) \cdot ((s_1+k)/d_1)}{e_1^2f_1^2} =
\frac{(s_2/d_2) \cdot ((s_2+k)/d_2)}{e_2^2f_2^2}.
$$
Let $\mathsf B_1=\{(s_1,s_2)\in\mathsf B\colon
e_1,f_1\le m^{1/3}\}$, $\mathsf B_2=\mathsf B\setminus \mathsf B_1$.
First, we will bound the cardinality of $\mathsf B_1$ and then of $\mathsf B_2$.

For every choice of $s_2,d_1,e_1,f_1$, the integer $s_1$ divides
$M = s_2(s_2+k)e_1^2f_1^2d_1^2$ (recall that
$ s_1(s_1+k)e_2^2f_2^2d_2^2=s_2(s_2+k)e_1^2f_1^2d_1^2 $), and
the number of divisors of $M$ is $O(M^\ep) = O(m^\ep)$, $\ep>0$
(see, for example, \cite[Theorem~315]{HW}). Furthermore, the number of different choices of
$s_2$ is $\simeq m$; there is only a bounded number of different choices of $d_1$
(since the integer $k$ is fixed and $d_1$ divides $k$); and the number of different choices
of $e_1$, as well as of $f_1$ does not exceed $m^{1/3}$. Thus,
$$
\card(\mathsf B_1)\lesssim_{\ep} m^{(5/3)+\ep}\,.
$$

To bound the cardinality of $\mathsf B_2$, we observe that if $(s_1,s_2)\in\mathsf B_2$,
then either $e_1>m^{1/3}$ or $f_1>m^{1/3}$.
Now,
\begin{multline*}
\card\bigl\{ (s_1,s_2)\in\mathsf B\colon e_1>m^{1/3}\bigr\}
\le \sum_{m^{1/3}<e_1\le(\lambda m)^{1/2}}\,\,\doublesum_{\substack{m\le s_1,s_2\le \lambda m \\ s_1=ue_1^2}}1 \\
\lesssim  \sum_{e_1 > m^{1/3}}\, \frac{m^2}{e_1^2} \lesssim m^{5/3}.
\end{multline*}
In the same way, $\card\bigl\{ (s_1,s_2)\in\mathsf B: f_1>m^{1/3}\bigr\}\lesssim m^{5/3}$.
\end{proof}

\begin{proof}[Proof of Theorem~\ref{t2}] Let $\xi$ be either $\xi_{\tt S}$ or $\xi_{\tt R}$.
We will show that $\xi$ is a Wiener sequence with the spectral measure
being the Lebesgue measure on
the unit circle\footnote{
Likely, this fact can be also extracted from results of the recent work of Najnudel~\cite{Najn}, at least, in the completely multiplicative case.}.
The rest follows from Theorem~\ref{t1}.

Fix $k\ge 1$ and split the sum
\[
\sum_{s\ge 1} \xi(s) \overline{\xi(s+k)}
\]
into the blocks $m^{20}\le s < (m+1)^{20}$, $m=1, 2, \ldots$.
By Lemma~\ref{l11}, for every $m\ge 1$, we have
$$
\bP\Bigl(\Bigl|\sum_{m^{20}\le s<(m+1)^{20}}\xi (s)\overline{\xi (s+k)}\, \Bigr| > m^{(15/16) \cdot 20} \Bigr) \lesssim m^{-20/8},
$$
and hence, by the Borel-Cantelli lemma, almost surely, for sufficiently large $m$,
$$
\Bigl|\sum_{m^{20}\le s<(m+1)^{20}}\xi (s)\overline{\xi (s+k)}\, \Bigr| \le m^{75/4}.
$$
Therefore, almost surely,
\begin{multline*}
\Bigl|\sum_{1\le s\le n} \xi (s)\overline{\xi (s+k)}\, \Bigr|
\le \sum_{1\le \ell\le [n^{1/20}]}\,
\Bigl|\, \sum_{\ell^{20}\le s<(\ell+1)^{20}}\, \xi (s)\overline{\xi(s+k)}\, \Bigr|\, + \,
\sum_{[n^{1/20}]^{20}\le s\le n}\, 1
\\ \lesssim \sum_{1\le \ell\le [n^{1/20}]}\, \ell^{75/4} + n^{19/20}
\lesssim n^{79/80}+n^{19/20}=o(n),\qquad n\to\infty,
\end{multline*}
proving the theorem.
\end{proof}


\bigskip
\medskip

\noindent J.B.:
School of Mathematics, Tel Aviv University, Tel Aviv, Israel
\newline {\tt benatar@mail.tau.ac.il}
\smallskip\newline\noindent{A.B.: Institut de Mathematiques de Marseille,
Aix Marseille Universit\'e, CNRS, Centrale Marseille, I2M, Marseille, France
\newline {\tt alexander.borichev@math.cnrs.fr}
\smallskip\newline\noindent M.S.:
School of Mathematics, Tel Aviv University, Tel Aviv, Israel
\newline {\tt sodin@tauex.tau.ac.il}
}


\begin{thebibliography}{A}


\bibitem{Bellow} A.~Bellow, V.Losert,
The weighted pointwise ergodic theorem and the individual ergodic theorem along subsequences.
Trans. Amer. Math. Soc. {\bf 288} (1985), 307--345. 


\bibitem{BNS} A.~Borichev, A.~Nishry, M.~Sodin, Entire functions of exponential type represented by pseudo-random and random Taylor series.
J. Anal.\ Math.\ {\bf 133} (2017), 361--396.

\bibitem{BSW} A.~Borichev, M.~Sodin, B.~Weiss,
Spectra of stationary processes on $\mathbb Z$. 50 years with Hardy spaces, 141--157,
Oper. Theory Adv. Appl. {\bf 261}, Birkh\"auser, 2018

\bibitem{CL} Y.~W.~Chen, J.~E.~Littlewood,
Some new properties of power series.
Indian J. Math. {\bf 9} (1967), 289--324.

\bibitem{CKM}
J.~Coquet, T.~Kamae, M.~Mend\`es France, 
Sur la mesure spectrale de certaines suites arithmétiques. 
Bull. Soc. Math. France {\bf 105} (1977), 369--384. 


\bibitem{Cord}
C.~Corduneanu, 
Almost periodic oscillations and waves. Springer, New York, 2009. 

\bibitem{EO}
A.~Eremenko, I.~Ostrovskii,
On the ``pits effect'' of Littlewood and Offord.
Bull.\ Lond.\ Math.\ Soc.\ {\bf 39} (2007), 929--939.

\bibitem{HW} G.~H.~Hardy, E.~M.~Wright,
An introduction to the theory of numbers.
Sixth edition. Revised by D.~R.~Heath-Brown and J.~H.~Silverman. With a foreword by Andrew Wiles. Oxford University Press, Oxford, 2008.


\bibitem{KZ} Z.~Kabluchko, D.~Zaporozhets,
Asymptotic distribution of complex zeros of random analytic
functions. Ann.\ Prob.\ {\bf 42} (2014), 1374--1395.

\bibitem{KN} L.~Kuipers, H.~Niederreiter,
Uniform distribution of sequences. Pure and Applied Mathematics. Wiley-Interscience, New York-London-Sydney, 1974.

\bibitem{Levin} {B.~Ya.~Levin},
Distribution of zeros of entire functions.
Translated from the Russian by R.P.Boas, J.M.Danskin, F.M.Goodspeed, J.Korevaar, A.L.Shields and H.P.Thielman,
Revised edition, American Mathematical Society,
Providence, R.I., 1980.

\bibitem{Levy} P.~L\'evy,
Sur la croissance des fonctions enti\`eres.
Bulletin de la Soci\'et\'e Math\'ematique de France,  {\bf 58}  (1930),  29--59, 127--149.

\bibitem{Littlewood}
J.~E.~Littlewood,
A ``pits effect'' for all smooth enough integral functions with a coefficient factor
$ \exp(n^2 \alpha\pi{\rm i})$, $\alpha = \frac12 (\sqrt5-1)$.
J. London Math.\ Soc.\ {\bf 43} (1968), 79--92.

\bibitem{Littlewood2}
J.~E.~Littlewood,
The ``pits effect'' for the integral function \\
$ f(z) = \sum \exp\{ -\vartheta^{-1}(n\log n - n) + \pi {\rm i}\alpha n^2\}z^n$,
$\alpha = \frac12\, (\sqrt{5}-1) $. \\
1969 Number Theory and Analysis (Papers in Honor of Edmund Landau) pp. 193--215. 
Plenum, New York.

\bibitem{LO}
J.~E.~Littlewood, A.~C.~Offord,
On the distribution of zeros and $a$-values of a random integral function. II.
Ann.\ of Math.\ (2) {\bf 49} (1948), 885--952; errata {\bf 50} (1949), 990--991.

\bibitem{Najn} J.~Najnudel,
On consecutive values of random completely multiplicative functions.
{\tt https://arxiv.org/abs/1702.01470}

\bibitem{Nassif} M.~Nassif,
On the behaviour of the function $ f(z)=\sum_{n=0}^\infty e^{\sqrt2\pi{\rm i} n^2}\, \frac{z^{2n}}{n!}$.
Proc.\ London Math.\ Soc.\ (2) {\bf 54} (1952), 201--214.


\bibitem{Tit} E.~C.~Titchmarsh, The theory of the Riemann zeta-function. Second edition. Edited and with a preface by D.~R.~Heath-Brown. The Clarendon Press, Oxford University Press, New York, 1986.

\bibitem{Wiener} N.~Wiener,
The Fourier integral and certain of its applications.
Cambridge University Press, 1933.


\bibitem{Wintner} A.~Wintner, Random factorizations and Riemann's hypothesis. Duke Math. J. {\bf 11} (1944), 267--275.



 \end{thebibliography}
\end{document}